\documentclass[12pt]{amsart}
\usepackage{latexsym, url}
\usepackage{amsthm}
\usepackage{amsmath}
\usepackage{amsfonts}
\usepackage{amssymb}
\usepackage[dvips]{graphicx}
\usepackage{xypic}
\input xy
\xyoption{all}

\newtheorem{theo}{Theorem}[section]
\newtheorem{lemma}[theo]{Lemma}
\newtheorem{assume}[theo]{Assumption}
\newtheorem{obs}[theo]{Observation}
\newtheorem{propo}[theo]{Proposition}

\newtheorem{defi}[theo]{Definition}
\newtheorem{coro}[theo]{Corollary}
\newtheorem{rem}[theo]{Remark}

\newtheorem{exam}[theo]{Example}
\newtheorem{exams}[theo]{Examples}

\newcommand\Surj{\operatorname{\it Surj}}
\newcommand\Inj{\operatorname{\it Inj}}

\newcommand\Alg{\operatorname{\bf Alg}}

\newcommand\Pos{\operatorname{\bf Pos}}

\newcommand\op{\operatorname{op}}
\newcommand\id{\operatorname{id}}
\newcommand\Id{\operatorname{Id}}

\newcommand\Set{\operatorname{\bf Set}}
\newcommand\Dist{\operatorname{\bf Dist}}
\newcommand\PMet{\operatorname{\bf PMet}}
\newcommand\Met{\operatorname{\bf Met}}
\newcommand\Gra{\operatorname{\bf Gra}}

\newcommand\colim{\operatorname{colim}}

\newcommand\ca{\mathcal {A}}

\newcommand\cc{\mathcal {C}}
\newcommand\cd{\mathcal {D}}
\newcommand\cf{\mathcal {F}}

\newcommand\ck{\mathcal {K}}
\newcommand\cl{\mathcal {L}}

\newcommand\cn{\mathcal {N}}

\newcommand\cs{\mathcal {S}}
\newcommand\ct{\mathcal {T}}

\newcommand\cx{\mathcal {X}}

\newcommand\cv{\mathcal {V}}

\newcommand\eps{\varepsilon}

 \newbox\noforkbox \newdimen\forklinewidth
\forklinewidth=0.3pt \setbox0\hbox{$\textstyle\smile$}
\setbox1\hbox to \wd0{\hfil\vrule width \forklinewidth depth-2pt
 height 10pt \hfil}
\wd1=0 cm \setbox\noforkbox\hbox{\lower 2pt\box1\lower
2pt\box0\relax}

\date{December 31, 2023}
 
\begin{document}
\title[Discrete equational theories]
{Discrete equational theories}
\author[J. Rosick\'{y}]
{J. Rosick\'{y}}
\thanks{Supported by the Grant Agency of the Czech Republic under the grant 22-02964S} 
 
\address{\newline J. Rosick\'{y}\newline
Department of Mathematics and Statistics,\newline
Masaryk University, Faculty of Sciences,\newline
Kotl\'{a}\v{r}sk\'{a} 2, 611 37 Brno,\newline
 Czech Republic}
\email{rosicky@math.muni.cz}

\begin{abstract}
On a locally $\lambda$-presentable symmetric monoidal closed category $\cv$, $\lambda$-ary enriched equational theories correspond to enriched monads preserving $\lambda$-filtered colimits. We introduce  discrete
$\lambda$-ary enriched equational theories where operations are 
induced by those having discrete arities (equations are not required to have discrete arities) and show that they correspond to enriched monads preserving preserving $\lambda$-filtered colimits and surjections. Using it, we prove enriched Birkhoff type theorems for categories of algebras of discrete theories. This extends known results from metric spaces and posets to general symmetric monoidal closed categories.

\end{abstract}
\maketitle
\section{Introduction}
Motivated by probabilistic programming, Mardare, Panangaden and Plotkin \cite{MPP1,MPP} introduced quantitative equations and developed universal algebra over metric spaces. In the recent paper
\cite{R}, we have related their quantitative equational theories to the theories of Bourke and Garner \cite{BG}. More papers were influenced
by \cite{MPP1,MPP}, in particular \cite{A} and \cite{MU}. In the first one, Ad\' amek showed that $\omega_1$-basic quantitative equational theories correspond to $\aleph_1$-ary enriched monads on metric spaces preserving surjections. In \cite{MU}, Milius and Urbat clarified the Birkhoff-type theorems stated in \cite{MPP}.

In this paper, we work in a general symmetric monoidal closed category $\cv$ which is locally $\lambda$-presentable as a closed category. The leading examples are $\Pos$ (posets) and $\Met$ (metric spaces). Under discrete objects of $\cv$ we mean copowers of the monoidal unit $I$. 
In $\Pos$ these are discrete posets and in $\Met$ discrete spaces (having all distances $0$ or $\infty$). We introduce discrete theories whose operations are induced by those with a discrete arity. We show that discrete theories correspond to $\lambda$-ary enriched monads on $\cv$ preserving surjections (i.e., morphisms for which is $I$ projective). As a special case, we get the result of Ad\' amek \cite{A}. Under mild assumptions, surjections form a left part of a factorization system $(\Surj,\Inj)$ on the underlying category $\cv_0$ of $\cv$. We assume that $\cv_0$ is $\Inj$-locally $\mu$-generated for $\mu\leq\lambda$ in the sense of \cite{DR} and prove Birkhoff's type theorems for categories of algebras of discrete theories. Metric spaces are locally $\aleph_1$-presentable and also $\Inj$-locally $\aleph_0$-generated and our Birkhoff-type theorems yield those of \cite{MPP}.

The Appendix withdraws some claims of \cite[Section 5]{R}.
\vskip 1 mm
{\bf Acknowledgement.} The author is grateful to Ji\v r\'\i{} Ad\' amek for valuable discussions, to Jason Parker for pointing out that
\cite[Proposition 5.1]{R} is false and to the anonymous referees for the careful examination of the paper and many valuable suggestions. In particular, they detected a gap in our original proof of Birkhoff type theorems.

\section{Background on algebraic theories}
Classical universal algebra starts with a \textit{signature} $\Sigma$ giving a set of operations equipped with arities which are finite cardinals. A $\Sigma$-algebra $A$ assigns to every $n$-ary operation $f$ a mapping $f_A:A^n\to A$. Then one defines \textit{terms} and \textit{equations}, interprets terms on an algebra and says when an equation is satisfied by an algebra. An \textit{equational theory} $E$ is a set of equations and an algebra satisfies $E$ if it satisfies all equations from $E$. The forgetful functor $U:\Alg(E)\to\Set$ from the category $\Alg(E)$ of algebras satisfying $E$ has a left adjoint $F$ and $\Alg(E)$ is equivalent to the category $\Alg(T)$ of algebras for the monad $T=UF$.
The algebra $Fn$ consists of equivalence classes of $n$-ary terms and morphisms $Fm\to Fn$ give $m$-tuples of $n$-ary terms. We can consider them as $(n,m)$-ary operations where the \textit{input arity} $n$ is the usual arity and the \textit{output arity} $m$ is the multiplicity. Then the superposition of terms is replaced by the composition of these operations. If $\cn$ is the full subcategory of $\Set$ consisting of finite cardinals and $\ct$ the full subcategory of $\Alg(E)$ consisting of free algebras on finite cardinals then the domain-codomain restriction of $F$ gives an identity-on-objects functor $J:\cn\to\ct$. The dual $\ct^{\op}$ of $\ct$ is a Lawvere theory whose algebras $A$ are functors $\hat{A}:\ct^{\op}\to\Set$ preserving finite products. This is the same as $\hat{A}J^{\op}=\Set(K-, A)$ where $K:\cn\to\Set$ is the inclusion and $A=\hat{A}(1)$. Then $\Alg(E)$ is equivalent to the category of algebras of $\ct^{\op}$. Hence classical universal algebra can be equivalently captured by equational theories, finitary monads (= monads preserving filtered colimits), or Lawvere theories. All this is well known due to Birkhoff, Lawvere and Linton (see \cite{ARV1}).

In ordered universal algebra, given a signature $\Sigma$, a $\Sigma$-algebra is a poset $A$ equipped with monotone mappings $f_A:A^n\to A$ for every $n$-ary operation $f$ from $\Sigma$. But, instead of equations, one takes inequations $p\leq q$ of terms (see \cite{Bl}. 
An algebra satisfies this inequation if $p_A\leq q_A$ in the poset 
of monotone mappings $A^n\to A$.
It is natural to take enriched signatures where arities are finite posets. Then the resulting enriched inequational theories correspond to finitary enriched monads on $\Pos$ (\cite{AFMS}). Recall that the category $\Pos$ of posets and monotone mapping is locally finitely presentable as a cartesian closed category. The free algebra $FX$ on a finite poset $X$ consists of equivalence classes of $X$-ary terms
and morphisms $FY\to FX$ can be taken as $(X,Y)$-ary operations. The inequation $p\leq q$ of $X$-ary terms, which is called an \textit{inequation in the context}  $X$, then means that the pair $p,q$ of terms is an $(X,2)$-ary operation where $2$ is a two-element chain. 
In this way, inequational theories can be replaced by equational theories in $(X,Y)$-ary operations (see \cite[Example 4.11(1)]{R}). If $\cf$ is the full subcategory of $\Pos$ consisting of finite poset and $\ct$ the full subcategory of $\Alg(E)$ consisting of free algebras on finite posets then the domain-codomain restriction of $F$ gives an identity-on-objects enriched functor $J:\cf\to\ct$. The dual $\ct^{\op}$ of $\ct$ is an enriched Lawvere theory (\cite{P1}). Hence ordered universal algebra can be equivalently captured by inequational theories, enriched equational theories, finitary monads, or enriched Lawvere theories.

In metric universal algebra, one has to take into account that the category $\Met$ of metric spaces (distances $\infty$ are allowed) and nonexpanding maps is only locally $\aleph_1$-presentable as a symmetric monoidal closed category. Thus one has to take countable metric spaces as arities. Given such an enriched signature $\Sigma$, a $\Sigma$-algebra is a metric space $A$ equipped with nonexpanding mappings $f_A:A^X\to A$ for every $X$-ary operation $f$ from $\Sigma$. Now, instead of equations, one takes quantitative equations $p=_\eps q$ of terms where $\eps\geq 0$ is a real number (see \cite{MPP,MPP1}). An algebra satisfies this quantitative equation if $d(p_A,q_A)\leq\eps$ in the metric space of nonexpanding mappings $A^X\to A$. These equations are called \textit{basic quantitative equations} (in context $X$). This means that the pair $p,q$ of terms is an $(X,2_\eps)$-ary operation where $2_\eps$ is a two-element metric space with the distance $\eps$ between the two points. In this way, basic quantitative theories can be replaced by equational theories in $(X,Y)$-ary operations (see \cite[Example 4.11(2)]{R}). If $\cc$ is the full subcategory of $\Met$ consisting of countable metric spaces and $\ct$ the full subcategory of $\Alg(E)$ consisting of free algebras on countable metric spaces then the domain-codomain restriction of $F$ gives an identity-on-objects enriched functor $J:\cf\to\ct$. The dual $\ct^{\op}$ of $\ct$ is an enriched $\aleph_1$-ary enriched Lawvere theory and these theories correspond to enriched monads preserving $\aleph_1$-filtered colimits (see \cite{P1}). Hence metric universal algebra can be equivalently captured by basic quantitative theories, enriched equational theories, $\aleph_1$-ary monads, or enriched $\aleph_1$-ary Lawvere theories.

In general, let $\cv$ be a symmetric monoidal closed category with the unit object $I$ and the underlying category $\cv_0$. We assume that $\cv$ is locally $\lambda$-presentable as as a symmetric monoidal closed category which means that the underlying category $\cv_0$ is locally $\lambda$-presentable, the tensor unit $I$ is $\lambda$-presentable and $X\otimes Y$ is $\lambda$-presentable whenever $X$ and $Y$ are $\lambda$-presentable. We will denote by $\cv_\lambda$ the (representative) small, full subcategory consisting of $\lambda$-presentable objects.

Following \cite{BG}, let $\ca$ be a small, full, dense sub-$\cv$-category of $\cv$ with the inclusion $K:\ca\to\cv$. Objects of $\ca$ are called \textit{arities}. Then an $\ca$-\textit{pretheory} is an identity-on-objects $\cv$-functor $J:\ca\to\ct$. A $\ct$-\textit{algebra} is an object $A$ of $\cv$ together with a $\cv$-functor $\hat{A}:\ct^{\op}\to\cv$ whose composition with $J^{\op}$ is $\cv(K-,A)$ (in \cite{BG}, $\ct$-algebras are called concrete 
$\ct$-models). Every $\ca$-pretheory induces a $\cv$-monad $T:\cv\to\cv$ given by its $\cv$-category $\Alg(\ct)$ of algebras. Conversely, a $\cv$-monad $T$ induces an $\ca$-pretheory $J:\ca\to\ct$ where $\ct$ is the full subcategory of $\Alg(T)$ consisting of free algebras on objects from $\ca$ and $J$ is the domain-codomain restriction of the free algebra functor.
An $\ca$-pretheory is an $\ca$-\textit{theory} if it is given by its monad. Then $\hat{A}$ is the hom-functor $\Alg(\ct)(-,A)$ restricted on free algebras over $\ca$. Conversely, a monad $T$ is $\ca$-nervous if it is given by its $\ca$-theory. In this way, we get a one-to-one correspondence between $\ca$-theories and $\ca$-nervous monads (see \cite[Corollary 21]{BG}). 

Under a $\lambda$-\textit{ary $\cv$-theory} we will mean
a $\cv_\lambda$-theory. Following \cite{BG}, $\lambda$-ary 
$\cv$-theories correspond to $\cv$-monads on $\cv$ preserving
$\lambda$-filtered colimits. They are called $\lambda$-\textit{ary
$\cv$-monads}. Hence $\lambda$-ary monads are precisely $\cv_\lambda$-nervous monads.

\section{Surjections}
The underlying functor $\cv_0(I,-):\cv_0\to\Set$ has a left adjoint $-\cdot I$ sending a set $X$ to the coproduct $X\cdot I$ of $X$ copies of $I$ in $\cv_0$. Objects $X\cdot I$ will be called \textit{discrete}. Every object $V$ of $\cv$ determines a discrete object $V_0=\cv_0(I,V)\cdot I$ and morphism $\delta_V:V_0\to V$ given by the counit of the adjunction. Every morphism $f:V\to W$ determines the morphism $f_0=\cv_0(I,f)\cdot I$ between the underlying discrete objects.

A morphism $f:A\to B$ will be called a \textit{surjection} if $\cv_0(I,f)$ is surjective. Let $\Surj$ denote the class of all surjections in $\cv_0$ and let $\Inj$ be the class of morphisms of $\cv_0$ having the unique right lifting property with respect to every surjection.
Morphisms from $\Inj$ will be called \textit{injections}. Recall that $g$
is an injection iff for every surjection $f$ and every commutative square
$$
\xymatrix@C=4pc@R=3pc{
A \ar [r]^{u} \ar [d]_{f}& C \ar [d]^{g}\\
B \ar [r]_{v}& D
}
$$
there is a unique $t:B\to C$ such that $tf=u$ and $gt=v$.

\begin{lemma}\label{surj}
$\Surj$ is accessible and closed under $\lambda$-directed colimits in $\cv_0^\to$.
\end{lemma}
\begin{proof}
$f$ is a surjection iff it has the right lifting property with respect to $0\to I$.
Like in \cite[Proposition 3.3]{R2}, the result follows from \cite[Proposition 4.7]{AR3}.
\end{proof}

\begin{rem}\label{surj1}
{
\em
(1) Surjections are right-cancellable, i.e., if $gf$ is a surjection 
then $g$ is a surjection.

(2) Surjections are closed under products and stable under pullbacks (because they are given by the right lifting property with respect to $0\to I$).
}
\end{rem}

\begin{lemma}\label{fs}
$(\Surj,\Inj)$ is a factorization system in $\cv_0$ if and only if $\Surj$ is closed under colimits in $\cv_0^\to$.  
\end{lemma}
\begin{proof}
Necessity is evident. Conversely, following \ref{surj}, $\Surj$
has a small dense subcategory $\cs$. If $\Surj$ is closed under colimits in $\cv_0^\to$ then $\Surj$ is the closure under $\cs$ under colimits in $\cv_0^\to$. Following \cite[Theorem 2.2]{FR}, $(\Surj,\Inj)$ is a factorization system.
\end{proof}

\begin{lemma}\label{fs1}
Injections are monomorphisms.
\end{lemma}
\begin{proof}
Take an injection $g:C\to D$ and $u,v:B\to C$ such that $gu=gv$.
Consider a commutative square
$$
\xymatrix@C=4pc@R=3pc{
B+B \ar [r]^{(u,v)} \ar [d]_{\nabla}& C \ar [d]^{g}\\
B \ar [r]_{gu}& D
}
$$
where $\nabla$ is the codiagonal. Since $\nabla$ is a split epimorphism, it is a surjection and thus there is a unique diagonal $B\to C$. Hence $u=v$. We have proved that $g$ is a monomorphism.
\end{proof}

\begin{lemma}\label{fs3}
Assume that $(\Surj,\Inj)$ is a factorization system. Then 
\begin{enumerate}
\item $\Surj$ contains all strong epimorphisms, and
\item if $I$ is a generator in $\cv_0$ then $\Inj$ contains all strong monomorphisms.
\end{enumerate}
\end{lemma}
\begin{proof}
In every locally presentable category, (strong epi, mono) and (epi, strong mono) are factorization systems (see \cite[Proposition 1.61]{AR3}). Thus (1) follows from \ref{fs1}. For (2), it suffices to show that every surjection is an epimorphism, which follows from $I$ being a generator.
\end{proof}

Recall that $I$ is connected iff $\cv_0(I,-)$ preserves coproducts.

\begin{lemma}\label{fs4}
If $I$ is connected and every regular epimorphism is a surjection then
$(\Surj,\Inj)$ is a factorization system.
\end{lemma}
\begin{proof}
If $I$ is connected then $\Surj$ is closed under coproducts 
in $\cv_0^\to$. If regular epimorphisms are surjections then surjections are closed under coequalizers in $\cv_0^\to$. Indeed, let $f_0$ and $f_1$
be surjections and 
$$ 
	\xymatrix@=4pc{
		&  
	f_0	\ar@<0.5ex>[r]^{(u_0,v_0)}
		\ar@<-0.5ex>[r]_{(u_1,v_1)}& f_1  \ar[r]^{(u,v)} & f
	}
	$$ 
be a coequalizer in $\cv_0^\to$. Then 
$$ 
	\xymatrix@=4pc{
		&  
	B_0	\ar@<0.5ex>[r]^{v_0}
		\ar@<-0.5ex>[r]_{v_1}& B_1  \ar[r]^{v} & B
	}
	$$	
is a coequalizer in $\cv_0$, hence $v$ is a surjection. Following 
\ref{surj1}(1), $f$ is a surjection.
\end{proof}

The factorization system $(\Surj,\Inj)$ is
$\lambda$-\textit{convenient} in the sense of \cite{DR} if
\begin{enumerate}
\item $\cv$ is $\Surj$-cowellpowered, i.e., every object of $\cv$ has only a set of surjective quotients, and
\item $\Inj$ is closed under $\lambda$-directed colimits, i.e., every $\lambda$-directed colimit of injections
has the property that a colimit cocone 
\begin{enumerate}
\item  consists of injections, and
\item for every cocone of injections, the factorizing morphism is an injection.
\end{enumerate}
\end{enumerate}

If $I$ is a generator then surjections are epimorphisms and (1) follows from the fact that every locally presentable category is co-wellpowered (see \cite[Theorem 1.58]{AR3}).

\begin{exams}
{
\em
(1) Let $\Pos$ be the category of posets and monotone mappings. $\Pos$
is cartesian closed and $I$ is the one-element poset $1$. Surjections are surjective monotone mappings, i.e, epimorphisms. Injections are embeddings and $(\Surj,\Inj)$ is an $\omega$-convenient factorization system.

(2) Let $\Met$ be the category of (generalized) metric spaces (i.e., with distances $\infty$ allowed) and nonexpanding maps. $\Met$ is a symmetric monoidal closed category where $I$ is the one-element metric space $1$ and $A\otimes B$ has the underlying set $A\times B$ and the metric 
$$
d((a,b),(a',b'))=d(a,a')+d(b,b').
$$ 
Surjections are surjective nonexpanding mappings, injections are isometries  and 
$$
(\Surj,\Inj)
$$ 
is a factorization system (cf. \cite[Example 3.16(1)]{AR}). This factorization system is $\omega$-convenient (see \cite[Remark 2.5(2)]{AR}).
 
(3) Let $\Gra$ be the cartesian closed category of graphs (i.e., sets with a symmetric binary relation) and graph homomorphisms. Then $I$ is the one-element graph with a loop. Let $V$ be a point (the one-element graph without a loop) and $E$ an edge (the two-element graph without loops and with one edge). Then $I$ is the coequalizer of two morphisms $V\to E$. Since $f:E\to I$ is not surjective, \ref{fs3}(1) implies that $(\Surj,\Inj)$ is not a factorization system on $\cv$.
}
\end{exams}

\section{Discrete theories}

\begin{assume}\label{ass1} 
Throughout the rest of the paper, we assume that $\cv$ is locally $\lambda$-presentable as a symmetric monoidal closed category and 
that the functor $\cv_0(I,-):\cv_0\to\Set$ preserves $\lambda$-presentable objects.
\end{assume}

Then the object $X_0$ of $\cv$ is $\lambda$-presentable whenever $X$ is $\lambda$-presentable. 
  
\begin{defi}\label{discrete}
{
\em
We say that a $\lambda$-ary $\cv$-theory $\ct$ is 
\textit{discrete} if every morphism $JI\to JX$ of $\ct$ is a composition $J(\delta_X)f$ for some morphism $f:JI\to JX_0$.
}
\end{defi}

This means that operations of arity $X$ are induced by those of arity $X_0$.

\begin{exams}\label{compare}
{
\em
(1) $\Pos$ is locally finitely presentable as a cartesian closed category. Finitely presentable posets are finite posets, hence
$\Pos(I,-):\Pos\to\Set$ preserves finitely presentable objects. Let $\ct$
be a discrete finitary theory. Let $\Sigma$ be the finitary signature whose $n$-ary operation symbols are $\ct$-morphisms
$J1\to Jn$ where $n$ denotes the discrete poset with $n$ elements. 
Morphisms $J1\to JX$ are $X$-ary terms and, since $\ct$ is discrete,
they are the restrictions of usual $X_0$-ary operation symbols on $X$
(in the sense of \ref{discrete}).
Pairs $f\leq g$, where $f,g:J1\to JX$ yield inequations in context $X$
in the sense of \cite{AFMS}. Hence $\ct$ yields a set of inequations in context. The meaning of $f\leq g$ 
on a $\ct$-algebra $A$ is that $f_A(a_1,\dots,a_n)\leq g_A(a_1,\dots,a_n)$ provided that $a:n\to A$ factorizes through $X\to A$; here $n=X_0$. Hence $\ct$-algebras coincide with algebras in the sense of \cite{AFMS}. An example of a discrete finitary theory is the theory  of ordered monoids which are commutative for comparable elements, i.e., $x\cdot y=y\cdot x$ for $x\leq y$.

Conversely, given a finitary signature $\Sigma$ (with discrete arities), 
the free $\Sigma$-algebra $FX$ on a poset $X$ consists of $X_0$-ary terms where $t\leq t'$ if $t(x_1,\dots,x_n)$ and $t'(x'_1,\dots,x'_n)$ have the same shape (i.e., we can get $t'$ from $t$ by changing variables $x_1,\dots,x_n$ to $x'_1,\dots,x'_n$) and $x_i\leq x'_i$ in $X$ for $i=1,\dots,n$. Let $E$ be a set of inequations in context. Following
\cite[3.22]{AFMS}, free $E$-algebra $JX$ on a poset $X$ is a quotient
$q_X:FX\to JX$ of $FX$. If $X$ is finite, every morphisms $t:J1\to JX$
factorizes through $q_X$ as $t=q_X\bar{t}$. Hence $\bar{t}$ is an $X_0$-ary term and $q_{X_0}(\bar{t})$ yields a factorization of $t$ through
$J(\delta_X)$. Thus $J$ is a discrete $\cv$-theory.

 (2) $\Met$ is locally $\aleph_1$-presentable as a symmetric monoidal closed category. Since $\aleph_1$-presentable objects are the metric spaces having cardinality $<\aleph_1$, $\Met(I,-)$ preserves $\aleph_1$-presentable objects. Let $\ct$ be a discrete $\aleph_1$-ary theory. Like in (1), let $\Sigma$ be the signature whose 
$n$-ary operation symbols are $\ct$-morphisms $J1\to Jn$ where $n$ denotes the discrete metric space with $n$ elements where $n\leq\omega$. Morphisms $J1\to JX$ are $X$-ary terms and, since $\ct$ is discrete, they are the restrictions of usual $X_0$-ary operation symbols on $X$.
Pairs $f,g:J1\to JX$ where $f,g:J1\to JX$ and $d(f,g)\leq\eps$ yield
$\omega_1$-\textit{basic quantitative equations} (\cite{MPP}).
Hence $\ct$ yields an $\omega_1$-\textit{basic quantitative equational theory} of \cite{MPP} (= a set of $\omega_1$-basic quantitative equations). The meaning $d(f,g)\leq\eps$ on a $\ct$-algebra $A$ is that $d(f_A(a_1,\dots,a_n),g_A(a_1,\dots,a_n))\leq\eps$ provided that $a:n\to A$ factorizes through $X\to A$; here $n=X_0$. Hence $\ct$-algebras coincide with algebras in the sense of \cite{MPP}. 
 
Conversely, every $\omega_1$-basic quantitative equational theory $E$ (in a signature $\Sigma$ with discrete arities) yields a discrete $\aleph_1$-ary enriched theory $\ct$. Analogously to (1), the free
$\Sigma$-algebra on a countable metric space $X$ consists $X_0$-ary terms 
and $d(t,t')\leq\eps$  if $t(x_1,\dots,x_n)$ and $t'(x'_1,\dots,x'_n)$ have the same shape and $d(x_i,x'_i)\leq\eps$ in $X$ for $i=1,\dots,n$ (see \cite{MPP1}). This means that $d(t,t')=\infty$ if $t$ and $t'$ do not have the same shape. Following \cite[6.1]{MPP1}, free $E$-algebra $JX$ on a poset $X$ is a quotient $q_X:FX\to JX$ of $FX$. The rest is the same as in (1).
}
\end{exams}

\begin{lemma} 
Let $\ct$ be a $\lambda$-ary $\cv$-theory. The following conditions are equivalent:
\begin{enumerate}
\item $\ct$ is discrete,
\item every morphism $JY\to JX$ of $\ct$ with $Y$ discrete is a composition $J(\delta_X)f$ with $f:JY\to JX_0$ and
\item for every morphism $g:JY\to JX$ of $\ct$, there is a morphism
$f:JY_0\to JX_0$ such that $gJ(\delta_Y)=J(\delta_X)f$.
\end{enumerate}
\end{lemma}
\begin{proof}
Clearly (3)$\to$(2)$\to$(1).

(1)$\to$(2). Every discrete object is a coproduct of copies of $I$. Since
$\ct$ is an $\cv_\lambda$-theory, $J$ is given by free $\ct$-algebras and
thus it preserves coproducts. Hence (1) implies (2).

(2)$\to$(3). It suffices to apply (2) on the composition $gJ(\delta_Y)$.
\end{proof}

\begin{theo}\label{discrete1} 
A $\lambda$-ary $\cv$-theory $\ct$ is discrete if and only if its induced monad preserves surjections.
\end{theo}
\begin{proof}
I. Let $\ct$ be discrete, $U:\Alg(\ct)\to\cv$ be the forgetful $\cv$-functor and $F$ its left $\cv$-adjoint. Then the induced $\cv$-monad is $T=UF$. Consider a $\lambda$-presentable object $X$, $\delta_X:X_0\to X$ and $a:I\to UFX$. Since $\ct$ is discrete, the adjoint transpose $\tilde{a}:FI\to FX$ factorizes through $F\delta_X$, i.e.,
$\tilde{a}=F(\delta_X)b$ where $b:FI\to FX_0$. Hence  
 $a=UF(\delta_X)\tilde{b}$. We have proved that $UF(\delta_X)$ is surjective.

Consider an arbitrary $X$ in $\cv$ and express it as a $\lambda$-directed colimit of $\lambda$-presentable objects $X_m$ ($m\in M$). Since $I$ is $\lambda$-presentable, $X_0=\colim X_{m0}$ and $\delta_X=\colim\delta_{X_m}$. Hence $UF(\delta_X)=\colim UF(\delta_{X_m})$ and, since surjections are closed under $\lambda$-directed colimits (see \ref{surj}), $UF(\delta_X)$ is a surjection.

Consider a surjective morphism $f:Y\to X$ in $\cv$ and $f_0:Y_0\to X_0$ be its underlying morphism. Then $f\delta_Y=\delta_Xf_0$ and thus $UF(f)UF(\delta_Y)=UF(\delta_X)UF(f_0)$. Since epimorphisms in $\Set$ split, $f_0$ is a split epimorphism and $UF(f_0)$ is surjective. Hence $UF(f)$ is surjective. 

II. Let $T$ be a $\lambda$-ary $\cv$-monad on $\cv$ preserving surjections. Let $\ct$ be the corresponding $\lambda$-ary 
$\cv$-theory. We will show that $\ct$ is discrete. Consider $f:FI\to FX$ and $\tilde{f}:I\to UFX$ its adjoint transpose. Since $\delta_X$ is surjective and $T$ preserves surjections, $UF(\delta_X)$ is surjective. Thus there is $g:I\to UFX_0$ such that $UF(\delta_X)g=\tilde{f}$. Let $\tilde{g}:FI\to FX_0$ be the adjoint transpose of $g$. Then $f=F(\delta_X)\tilde{g}$ and thus $\ct$ is discrete.
\end{proof}

\begin{exams}
{
\em
(1) \ref{discrete1} (together with \ref{compare}) gives the result of J. Ad\' amek that sets of inequations in context in finitary signatures correspond to enriched finitary mo\-nads preserving surjections (see his talk "Finitary monads on $\Pos$" at the conference Category theory CT20-21).

(2) Similarly, \ref{discrete1} gives the result of J. Ad\' amek \cite{A} that $\omega_1$-basic quantitative equational theories (in signatures with discrete arities) correspond to enriched $\omega_1$-ary monads preserving surjections.
}
\end{exams}

\begin{rem}\label{discrete2}
{
\em
Let $\cd_\lambda$ consist of discrete $\lambda$-presentable objects. Then
$\cd_\lambda$-theories correspond to the $\lambda$-ary discrete Lawvere theories of \cite{P1}. Over $\Pos$ these theories are discrete in the sense of \ref{discrete} and inequations of terms are in discrete contexts.
Following \cite{ADV} $\cd_\omega$-theories correspond to finitary enriched monads preserving reflexive coinserters. Over $\Met$, $\cd_{\aleph_1}$-theories correspond to unconditional quantitative theories of \cite{MPP} and \cite{MPP1} and are discrete in the sense of \ref{discrete}. In general, we do not know whether discrete Lawvere theories are discrete in the sense of \ref{discrete}.
}
\end{rem}

\section{Birkhoff subcategories}
\begin{assume}\label{ass} 
Throughout this section, we assume, in addition to \ref{ass1}, that $\ct$ is a discrete $\lambda$-ary $\cv$-theory and $(\Surj,\Inj)$ is a proper $\mu$-convenient factorization system in $\cv_0$ where $\mu\leq\lambda$. 
\end{assume}

\begin{obs}\label{obsv}
{
\em
Recall that $(\Surj,\Inj)$ is proper if surjections are epimorphisms and injections are monomorphisms. Let $T$ be the $\cv$-monad induced by $\ct$. The underlying category $\Alg(\ct)_0$ is the category of algebras for the underlying monad $T_0$. Let $U_0:\Alg(\ct)_0\to\cv_0$ be the forgetful functor. We will say that a morphism $f:A\to B$ in $\Alg(\ct)_0$ is a surjection (injection) if $U(f)$ is a surjection (injection). Following \cite[Chapter 3, Proposition 4.17]{M}, (surjections, injections) form a proper factorization system on $\Alg(\ct)_0$. Here we need that $\ct$ is discrete because then, following \ref{discrete1}, $T_0$ preserves surjection. Given a morphism $f:A\to B$ of $\ct$-algebras and 
$$
U_0A \xrightarrow{\ e } C \xrightarrow{\ m } U_0B 
$$ 
$(\Surj,\Inj)$-factorization of $U_0(f)$ then $f$ factorizes as
$$
A \xrightarrow{\ \overline{e} } \overline{C} \xrightarrow{\ \overline{m} } B 
$$ 
where $U_0\overline{C}=C$, $U_0(\overline{e})=e$ and $U_0(\overline{m})=m$.

Since $\Alg(\ct)_0$ is locally presentable, (strong epimorphisms, monomorphisms) is a factorization system on $\Alg(\ct)_0$ (see \cite[1.61]{AR3}). Since injections are monomorphisms, strong epimorphisms in $\Alg(\ct)_0$ are surjections.
}
\end{obs}

Following \cite[Chapter 3, 3.1]{M}, a \textit{Birkhoff subcategory} of $\Alg(\ct)_0$ is a full replete subcategory of $\Alg(\ct)_0$ closed under products, subalgebras and $U_0$-split quotients. 
Here, a morphism $g:K\to L$ in $\Alg(\ct)_0$ is $U_0$-split if $U_0(f)$ is a split epimorphism.   Hence, for a Birkhoff subcategory $\cl$ of $\Alg(\ct)_0$, the reflections $\rho_K:K\to K^\ast$ are strong epimorphisms.

\begin{lemma}\label{refl}
A Birkhoff subcategory of $\Alg(\ct)_0$ is $\cv$-reflective in $\Alg(\ct)$.
\end{lemma}
\begin{proof}
Let $\cl$ be a Birkhoff subcategory of $\Alg(\ct)$. Since $\cl_0$ is reflective in $\Alg(\ct)_0$, it suffices to show that $\cl$ is closed in $\Alg(\ct)$ under cotensors (see \cite[6.7.6]{B}. Consider $V$ in $\cv$ and $A$ in $\cl$. Since the cotensor functor $[V,-]:\Alg(\ct)\to\Alg(\ct)$ has a left $\cv$-adjoint $V\otimes -$ 
(see \cite[6.5.6]{B}), and $\delta_V:V_0\to V$ is an epimorphism,  $[\delta_V,A]:[V,A]\to [V_0,A]$ is a monomorphism. Since $[V_0,A]$ is a power of $A$, $[V,A]$ is in $\cl$.
\end{proof}

Hence Birkhoff subcategories of $\Alg(\ct)_0$ are $\cv$-subcategories and,
in what follows, we will take them as subcategories of $\Alg(\ct)$.
The following definition goes back to \cite{H,HR}.

\begin{defi}\label{equation}
{
\em
An $\ct$-\textit{equation} $p=q$ is a pair of morphisms $p,q:FY\to FX$ in $\Alg(\ct)$. A $\ct$-\textit{equational theory} $E$ is a set of $\ct$-equations.

A $\ct$-algebra $A$ \textit{satisfies} a $\ct$-equation $p=q$ if
$hp=hq$ for every morphism $h:FX\to A$. It satisfies a $\ct$-equational theory $E$ if it satisfies all equations of $E$.

$\Alg(E)$ will be the full subcategory of $\Alg(\ct)_0$ consisting of $\ct$-algebras satisfying all equations from $E$.
}
\end{defi}

\begin{propo}\label{birkhoff0} 
$\Alg(E)$ is a Birkhoff subcategory of $\Alg(\ct)$ for every $\ct$-equational theory $E$. 
\end{propo}
\begin{proof}
$\Alg(E)$ is clearly closed under products and subalgebras. Let 
$f:A\to B$ be a $U_0$-split quotient of an algebra $A$ satisfying $E$, i.e., there is $s:UB\to UA$ such that $U(f)s=\id_{UB}$. Let $p,q:FY\to FX$ give an equation $p=q$ from $E$. Consider $h:FX\to B$ and $\tilde{h}:X\to UB$ be the adjoint transpose of $h$. Let $g=s\tilde{h}$ and $\tilde{g}:FX\to A$ be the adjoint transpose of $g$. Since $U(f)g=U(f)s\tilde{h}=\tilde{h}$, we have $f\tilde{g}=h$. Since $\tilde{g}p=\tilde{g}q$, we get 
$$
hp=f\tilde{g}p=f\tilde{g}q=hq.
$$ 
Thus $B$ satisfies $E$.
\end{proof}
 
Recall that an object $V$ in $\cv$ is $\mu$-\textit{generated} with respect to $\Inj$ if $\cv_0(V,-):\cv_0\to\Set$ preserves $\mu$-directed colimits of injections.

\begin{defi}
{
\em
We say that a morphism $f:A\to B$ in $\cv_0$ is a $\mu$-\textit{pure epimorphism} if it is projective with respect to $\mu$-generated objects. Explicitly, for every $\mu$-generated object $X$, all morphisms $X\to B$ factor through $f$.  
}
\end{defi} 

\begin{rem}
{
\em
(1) For $\mu=\lambda$, this concept was introduced in \cite{AR1}.

(2) Every $\mu$-pure epimorphism $f:A\to B$ is an epimorphism. Indeed,
assume that $uf=vf$ for $u,v:B\to C$. Consider  a $\mu$-generated object $X$ and $g:X\to B$. Since $g$ factors through $f$, we have $ug=vg$.
Thus $u=v$.
 
(3) Every split epimorphism is $\mu$-pure. 

(4) A morphism $f:A\to B$ $\omega$-pure in $\Met$ iff it is $\omega$-reflexive in the sense of \cite{MPP}. 
}
\end{rem}

\begin{defi}
{
\em
We say that $\cl$ is a $\mu$-\textit{Birkhoff subcategory} of $\Alg(\ct)$ if it is a Birkhoff subcategory closed under quotients $f:A\to B$ such that $Uf$ is $\mu$-pure.
}
\end{defi}
Thus $\cl$ is a $\mu$-Birkhoff subcategory of $\Alg(\ct)$ if and only if it is a full subcategory closed under products, subalgebras and $\mu$-pure quotients.

Recall that $\cv_0$ is $\Inj$-\textit{locally $\mu$-generated} if it has a set $\cx$ of $\mu$-generated objects with respect to $\Inj$ such that every object of $\cv_0$ is a $\mu$-directed colimit of objects from $\cx$ and injections (see \cite{DR}).

\begin{rem}
{
\em
(1) Following \cite[Remark 2.17]{DR}, $\cv_0$ is $\Inj$-locally $\lambda$-generated.  

(2) $\Met$ is $\Inj$-locally $\omega$-generated because finite metric spaces are $\omega$-generated with respect to $\Inj$ (following 
\cite[Remark 2.5(2)]{AR}).
}
\end{rem}

The following definition and theorem were motivated by \cite{MU}.

\begin{defi}
{
\em
A $\ct$-equation $p=q$, where $p,q:FY\to FX$, will be called $\mu$-\textit{clustered} if $X$ is a coproduct of $\mu$-generated objects.
}
\end{defi}

\begin{theo}\label{birkhoff}
Assume \ref{ass} and, moreover, let $\cv_0$ be $\Inj$-locally $\mu$-generated. Then $\cl$ is a $\mu$-Birkhoff subcategory of $\Alg(\ct)$ if and only if $\cl=\Alg(E)$ where all equations from $E$ are $\mu$-clustered.  
\end{theo}
\begin{proof}
I. Necessity: Let $E$ consist of $\mu$-clustered equations. We have to show that $\Alg(E)$ is closed in $\Alg(\ct)$ under quotients $f:A\to B$ such that $Uf$ is $\mu$-pure. So, let $f:A\to B$ be such a homomorphism. Consider an equation $p=q$ from $E$ where $p,q:FY\to FX$ and $X=\coprod_i X_i$ with $X_i$ $\mu$-generated. Let $h:FX\to B$ and $\tilde{h}:X\to UB$ be its adjoint transpose. If $u_i:X_i\to X$ are coproduct injections, then there are $v_i:X_i\to UA$ such that $U(f)v_i=\tilde{h}u_i$ for every $i$. We get $v:X\to UA$
such that $U(f)v=\tilde{h}$. Therefore
$$
hp=f\tilde{v}p=f\tilde{v}q=hq.
$$
Hence $B$ satisfies $E$.

II. Sufficiency: Conversely, let $\cl$ be a $\mu$-Birkhoff subcategory of $\Alg(\ct)$. Following \ref{refl}, $\cl$ is $\cv$-reflective in $\Alg(\ct)$. Moreover, following \ref{obsv}, the reflections $\rho_K:K\to K'$ are strong epimorphisms. 

Let $U':\cl\to\cv_0$ be the restriction of $U_0$ on $\cl$ and $F'$ its left adjoint. Let $E$ be given by pairs $(p,q)$ where $p,q:FY\to FX$ with $X=\coprod X_i$, where $X_i$ are $\mu$-generated with respect to $\Inj$ $X_i$, such that every $A\in\cl$ satisfies the equation 
$p=q$. This is equivalent to $\rho_{FX}p=\rho_{FX}q$. Clearly $\cl\subseteq\Alg(E)$. Following I., $\Alg(E)$ is a $\mu$-Birkhoff subcategory of $\Alg(\ct)$. Hence $\Alg(E)$ is $\cv$-reflective in $\Alg(\ct)$ and the reflections $\rho'_K:K\to K''$ are strong epimorphisms.

Let $U'':\Alg(E)\to\cv_0$ be the restriction of $U$ on $\Alg(E)$ and $F''$ its left adjoint. We get $\rho':F\to F''$ and $\tau:F''\to F'$ such that $\tau\rho'=\rho_F$. Since both $\cl$ and $\Alg(E)$ are monadic, it suffices to prove that $\tau$ is a natural isomorphism.  

Consider an arbitrary $X$ in $\cv$ and express $UFX$ as a $\mu$-directed colimit 
$$
(z_m:Z_m\to UFX)_{m\in M}
$$ 
of injections between objects $Z_m$  which are $\mu$-generated with respect to $\Inj$. Let $\eta_Z:Z\to UFZ$, $Z$ in $\cv$, be the adjunction units and $\tilde{z}_m: FZ_m\to FX$ be given as $U(\tilde{z}_m)\eta_{Z_m}=z_m$. Let 
$$
t_X:F(\coprod Z_m)\to FX
$$ 
be the induced morphism, i.e., $t_XF(u_m)=\tilde{z}_m$ where 
$(u_m:Z_m\to \coprod Z_m)_{m\in M}$ is the coproduct. Then $U(t_X)$ is a $\mu$-pure epimorphism. Indeed, every morphism $f:Z\to UFX$ with $Z$ $\mu$-generated with respect to $\Inj$ factors through $z_m$ for some $m\in M$ as $f=z_mg$ where $g:Z\to Z_m$. Hence 
$$
\tilde{f}=\tilde{z}_mF(g)=t_XF(u_m)F(g)
$$
and thus
$$
f=U(\tilde{f})\eta_Z=U(t_X)UF(u_mg)\eta_Z.
$$

We have to prove that $\tau_X$ is an isomorphism for every $X$.
Consider the commutative diagram
$$  
  \xymatrix@=4pc{
    FUF''X \ar[r]^{\eps_{F''X}} & F''X  \\
    F\coprod Z_m \ar [u]^{t_{UF''X}} \ar [r]^{\rho'_{\coprod Z_m}} & 
F''\coprod Z_m \ar[u]_{s} 
  }
  $$
where $\eps:FU\to\Id$ is the counit of the adjunction $F\dashv U$.
The arrow $s$ is given by $\rho'_{\coprod Z_m}$ being the reflection of $F\coprod Z_m$ to $\Alg(E)$. Since $U(\eps_{F''X})$ is a split epimorphism and $U(t_{UF''X})$ is a $\mu$-pure epimorphism, the composition $U(\eps_{F''X}t_{UF''X})$ is a $\mu$-pure epimorphism. Thus $U(s)$ is a $\mu$-pure epimorphism. Thus it suffices to show that $F''\coprod Z_m$ belongs to $\cl$, i.e., that $\tau_Z$ is an isomorphism where $Z=\coprod Z_m$. 

Since $\rho_Z$ is a strong epimorphism, $\tau_Z$ is a strong epimorphism. Thus it suffices to show that $\tau_Z$ is a monomorphism. Consider $p,q:FY\to F''Z$ such that $\tau_Zp=\tau_Zq$. We have to show that $p=q$. Consider the pullback
$$  
  \xymatrix@=4pc{
    FZ\times FZ \ar[r]^{\rho'_Z\times\rho'_Z} & F''Z\times F''Z  \\
    P \ar [u]^{r} \ar [r]^{e} & FY \ar[u]_{\langle p,q\rangle}  
  }
  $$
and take the compositions $p',q':P\to FZ$ of the product projections with $r$. Let $p'',q'':FUP\to FZ$ be the compositions of $p'$ and $q'$ with $\eps_P:FUP\to P$. Since 
$$
\rho_{FZ}p''=\tau_Z\rho'_Zp''=\tau_Zpe\eps_P=\tau_Zqe\eps_P=\rho_{FZ}q'',
$$
$p''=q''$ belongs to $E$. Thus we have 
$\rho'_Zp''=\rho'_Zq''$,
i.e., $pe=qe$. Since $\rho'_Z$ is a strong epimorphism, it is a surjection (see \ref{obsv}). Following \ref{surj1}, $e$ is a surjection, hence an epimorphism. Thus $p=q$ and we have proved that $\tau_Z$ is a monomorphism.
\end{proof}

\begin{defi}\label{connected}
{
\em
We will call a category $\ck$ \textit{strongly connected} if for every pair of objects $K$ and $K'$ of $\ck$, where $K'$ is not initial, there is a morphism $K\to K'$.  
}
\end{defi}

\begin{lemma}\label{connected1}
The following conditions are equivalent:
\begin{enumerate}
\item $\ck$ is strongly connected, and 
\item for a coproduct $\coprod_{m\in M} K_m$ of non-initial objects 
$K_m$, coproduct components $u_m:K_m\to\coprod K_m$ are split monomorphisms.
\end{enumerate}
\end{lemma}
\begin{proof}
If $\ck$ is strongly connected and $u_m:K_m\to\coprod K_m$ a coproduct of non-initial objects, then there is a cocone $v_m:K_m\to K_n$ for every $n$ with $v_n=\id_{K_n}$. Hence $u_n$ is a split monomorphism.

Conversely, (2) applied to $K\coprod K'$ yields $K\to K'$.
\end{proof}

\begin{rem}\label{connected2}
{
\em
In \ref{connected1}, all subcoproduct morphisms 
$\coprod_{m\in N} X_m\to\coprod_{m\in M} X_m$, $N\subseteq M$, are split mono\-morphisms. In fact, 
$\coprod_{m\in M} X_m=(\coprod_{m\in N} X_m)\coprod(\coprod_{m\notin N}X_m)$.
}
\end{rem}

\begin{theo}\label{birkhoff1}
Assume \ref{ass} and, moreover, let $\cv_0$ be $\Inj$-locally $\mu$-generated, strongly connected and let $U$ preserve $\mu$-directed colimits of injections. Then $\cl$ is a $\mu$-Birkhoff subcategory of $\Alg(\ct)$ if and only if $\cl=\Alg(E)$ where all equations $p=q$ from $E$ have $p,q:FY\to FX$ with $X$ and $Y$ being $\mu$-generated with respect to $\Inj$. 
\end{theo}
\begin{proof}
Sufficiency is I. of the proof of \ref{birkhoff}. Conversely, let $\cl$
be a $\mu$-Birkhoff subcategory of $\Alg(\ct)$.
Following \ref{birkhoff}, $\cl=\Alg(E)$ where all equations $p=q$ from $E$ have $p,q:FY\to FX$ with $X$ being a coproduct of objects $\mu$-generated with respect to $\Inj$. Express $Y$ as a $\mu$-directed colimit $(y_m:Y_m\to Y)_{m\in M}$ of objects $Y_m$ $\mu$-generated with respect to $\Inj$. Then $F(y_m):FY_m\to FY$ is a $\mu$-directed colimit and $\rho_{FX}p=\rho_{FX}q$ if and only if $\rho_{FX}pF(y_m)=\rho_{FX}qF(y_m)$ for every $m\in M$. Thus we have reduced our equations to $p=q$  where $p,q:FY\to FX$ with $X$ being a coproduct of objects $\mu$-generated with respect to $\Inj$ and $Y$ being $\mu$-generated with respect to $\Inj$.  

Express $X=\coprod_{m\in M}X_m$ as a $\mu$-directed colimit of
subcoproducts $X_N=\coprod_{m\in N} X_m$ where $|N|<\mu$ and $X_m$ are
non-initial. Following \ref{connected2}, $X$ is a $\mu$-directed colimit $x_N:X_N\to X$ of split monomorphisms $x_{NN'}:X_N\to X_{N'}$ where $N\subseteq N'$. Hence $UFX$ is a $\mu$-directed colimit $UF(x_N):UFX_N\to UFX$ of split monomorphisms $UF(x_{NN'}):UFX_N\to UFX_{N'}$. 
Following \ref{fs3}, $UF(x_{NN'})$ are injections. Let $\tilde{p},\tilde{q}:Y\to UFX$ be the adjoint transposes of $p,q:FY\to FX$. Since $Y$ is $\mu$-generated with respect to $\Inj$,  there is $N\subseteq M$, $|N|<\mu$ and $p',q':Y\to UFX_N$ such that $\tilde{p}=x_Np'$ and $\tilde{q}=x_Nq'$.

Now, $p=q$ is an equation from $E$, iff $U(\rho_{FX})\tilde{p}=U(\rho_{FX})\tilde{q}$.
Since $x_N$ is a split monomorphism, this is equivalent to $U(\rho_{FX_N})p'=U(\rho_{X_N})q'$, hence to $\rho_{X_N}p^\ast=\rho_{X_N}q^\ast$ where $p^\ast,q^\ast:FY\to FX_N$ are adjoint transposes of $p',q'$. Since $X_N$ is $\mu$-generated with respect to $\Inj$ (see \cite[Lemma 2.13]{DR}), the proof is finished.
\end{proof}

\begin{rem}\label{birkhoff2}
{
\em
Moreover, $\Alg(E)$ is closed in $\Alg(\ct)$ under $\mu$-directed colimits of injections.

Indeed, let $k_m:K_m\to K$ be a $\mu$-directed colimit of injections where
$K_m$ satisfy $E$. Consider an equation $p=q$  from $E$ where $p,q:FY\to FX$. Since $U$ preserves $\mu$-directed colimits of injections, $F$ preserves $\mu$-generated objects (see \cite[Lemma 3.11]{DR}). Hence 
a morphism $h:FX\to K$ factors through some $k_m:K_m\to K$, $h=k_mh'$.
Since $K_m$ satisfies $E$, $h'p=h'q$. Thus $hp=hq$ and $K$ satisfies $E$.
}
\end{rem}

\begin{rem}
{
\em
The assumption that $\ck$ is strongly connected is needed
in \ref{birkhoff1} because \ref{birkhoff2} is not valid for the category $\Set^{\Bbb N}$ of $\Bbb N$-sorted sets. Indeed, let $T$ be a finitary monad in $\Set^{\Bbb N}$. Since epimorphisms in $\Set^{\Bbb N}$ split, every Birkhoff subcategory of $(\Set^{\Bbb N})^T$ is $\omega$-Birkhoff.
Following \cite{ARV}, it does not need to be closed under directed colimits. (Take the full subcategory of $\Set^{\Bbb N}$ consisting of all
$(X_n)_{n\in\Bbb N}$ such that either $X_n=\emptyset$ for some $n$ or
$X_n=1$ for every $n$.)
}
\end{rem}

As a consequence of \ref{birkhoff1}, we get the Birkhoff theorems over $\Met$ from \cite{MPP}, see \cite[Theorem 2.14]{A}. Here, under 
an $\omega$-ary theory over $\Met$ we mean an $\cf$-theory where $\cf$ consists of finite metric spaces, i.e., of finitely generated metric spaces with respect to $\Inj$. 

Manes \cite{M} defined a $\Surj$-\textit{Birkhoff subcategory} $\cl$ of $\Alg(\ct)_0$ as a full replete subcategory closed under products, $\Inj$-subalgebras and $U_0$-split quotients. Here, a morphism $g:K\to L$ in $\Alg(\ct)_0$ is an $\Inj$-subalgebra of $L$ if $g$ is an injection. Since the monad $T$ preserves surjections, \cite[Chapter 3, Theorem 4.23]{M} shows that a full subcategory $\cl$ of $\Alg(\ct)_0$ is a $\Surj$-Birkhoff subcategory if and only if it is a reflective subcategory such that $U_0(\rho_K)$ is a surjection for every reflection $\rho_K:K\to K^\ast$. Following \cite[Chapter 3, Theorem 3.3]{M}, $\cl$ is monadic over $\cv_0$. Hence every Birkhoff sub\-category of $\Alg(\ct)_0$ is a $\Surj$-Birkhoff subcategory. 

\begin{coro}
Let $\Sigma$ be a $\mu$-ary (discrete) signature where  $\mu=\omega$ or 
 $\mu=\omega_1$. Then a class of quantitative $\Sigma$-algebras is a $\mu$-Surj-Birkhoff subcategory of $\Alg(\Sigma)$ if and only it is given by an $\mu$-basic quantitative equational theory.
\end{coro}
\begin{proof}
At first, we observe that, for a discrete $\mu$-ary $\Met$-theory $\ct$, the functor $U:\Alg(\ct)\to\Met$ preserves $\mu$-directed colimits of isometries. For $\mu=\omega_1$ it is evident and, for $\mu=\omega$, it follows from \cite[Theorem 3.19]{R}.

Following \ref{compare}(2), $\Alg(\Sigma)\cong\Alg(\ct)$ for a discrete $\mu$-ary $\Met$-theory $\ct$. Given a $\mu$-basic quantitative equational theory $E$ then, following \ref{compare}(2) and I. of the proof of \ref{birkhoff}, $\Alg(E)$ is a $\mu$-$\Surj$-Birkhoff subcategory of $\Alg(\Sigma)$.

Conversely, let $\cl$ be a $\mu$-$\Surj$-Birkhoff subcategory of $\Alg(\Sigma)$.  We extend $\Sigma$ to a new signature $\Sigma^\ast$ by adding, for every $\eps>0$, binary operations $f_\eps$ and $g_\eps$   satisfying the quantitative equations
\begin{enumerate}
\item $f_\eps(x,y)=_\eps g_\eps(x,y)$ and
\item $x=_\eps y\vdash f_\eps(x,y)=x,\quad x=_\eps y\vdash g_\eps(x,y)=y.$
\end{enumerate}
Then subalgebras of $\Sigma^\ast$-algebras satisfying (1) and (2) are $\Inj$-subalgebras.
Indeed, let $B$ be a subalgebra of a $\Sigma^\ast$-algebra $A$ satisfying (1) and (2) and consider $a,b\in B$ such that $d_B(a,b)>d_A(a,b)$. Following (2) for $\eps=d_A(a,b)$,
we have 
$$
(f_\eps)_B(a,b)=(f_\eps)_A(a,b)=a,\quad (g_\eps)_B(a,b)=(g_\eps)_A(a,b)=b,
$$
which contradicts (1).
 
Let $\cl^\ast$ consist of all $\Sigma^\ast$-algebras satisfying (1) and (2) whose $\Sigma$-reduct is in $\cl$. Then $\cl^\ast$ is a $\mu$-Birkhoff subcategory of $\Alg(\Sigma^\ast)$. Let $\ct^\ast$ be the discrete $\mu$-ary $\Met$-theory given by $\Sigma^\ast$, (1) and (2). Following \ref{birkhoff1}, $\cl^\ast=\Alg(E^\ast)$ where equations $p=q$ from $E^\ast$ have
$p,q:F^\ast Y \to F^\ast X$ with $X$ and $Y$ being $\mu$-generated with respect to $\Inj$. Let $E'$ consist of those equations from $E^\ast$ which do not contain the added operations $f_\eps$ and $g_\eps$. Clearly,
$\Alg(E')\subseteq\cl$. Consider $A\in\cl$ and interpret $f_\eps$ and $g_\eps$ on $A$ as follows: $(f_\eps)_A(x,y)=x$ and $(g_\eps)_A(x,y)=y$
if $d(x,y)=\eps$ and $(f_\eps)_A=(g_\eps)_A$ otherwise. Since we get
an $E^\ast$-algebra, $\cl\subseteq\Alg(E')$.
\end{proof}

\section{Appendix}
\cite[Proposition 5.1]{R} claims that finite products commute with reflexive coequalizers in $\Met$. Equivalently, that the functor
$$
X\times -:\Met\to\Met
$$
preserves reflexive coequalizers. Since this functor preserves coproducts,
the preservation of reflexive coequalizers is equivalent to the preservation of all colimits, i.e., following the Special Adjoint Functor Theorem, to the cartesian closedness of $\Met$. Thus \cite[Proposition 5.1]{R}  is not true and I am grateful to Jason Parker for pointing this out. 

In fact, \cite[Proposition 5.1]{R} proves that the functor 
$$
X\times -:\Dist\to\Dist
$$
preserves reflexive coequalizers where $\Dist$ is the category of distance spaces and nonexpanding mappings. Recall that a \textit{distance space} is equipped with a metric $d:X\to[0,\infty]$ satisfying $d(x,y)=d(y,x)$ and $d(x,x)=0$. The category $\Dist$ is cartesian closed (\cite{ARe}). 

The category $\PMet$ of pseudometric spaces (i.e., with the triangle inequality added) is a reflective subcategory of $\Dist$; the reflection of $(X,d)$ is obtained by the \textit{pseudometric modification} $d^\ast$ of $d$:
$$
d^\ast(x,z)={\rm inf}\{\sum_{i=0}^{n-1}d(y_i,y_{i+1})\,|\, n\geq1, y_i\in X, y_0=x, y_n=z\}.
$$
Hence $\Met$ is a reflective subcategory of $\Dist$.
We have just explained that the reflector $F:\Dist\to\Met$ cannot preserve finite products. We will show that it does not even preserve finite powers.

\begin{exam}\label{power}
{
\em
Consider the distance spaces $A_0=\{x,y,z\}$ and $A_1=\{u,v,w\}$
where $d(x,y)=d(v,w)=2$, $d(y,z)=d(u,v)=1$ and $d(x,z)=d(u,w)=\infty$. Let $A$ be the coproduct $A_0+A_1$. Then the metric modification $FA$ only
changes $d(x,z)$ and $d(u,w)$ to $3$. Hence the distance of $[x,u]$ and 
$[z,w]$ in $FA\times FA$ is $3$.

In $A\times A$, we have $(d[x,u],[y,v])=d([y,v],[z,w])=2$, hence
$[d[x,u],[z,w])=4$ in $F(A\times A)$.
}
\end{exam}

The distance space $A_0$ above is the quotient in $\Dist$ of the coproduct $A_{00}+A_{01}+A_{02}$ of metric spaces $A_{00}=\{x,y\}$, $A_{01}=\{y',z\}$ and $A_{02}=\{x',z'\}$ where $d(x,y)=2$, $d(y',z)=1$ and $d(x',z')=\infty$ modulo the equivalence relation $x\sim x'$,
$y\sim y'$ and $z\sim z'$. Thus $FA_0$ is the quotient of $A_{00}+A_{01}+A_{02}$ modulo this equivalence relation in $\Met$.
Similarly $A_1$ is the quotient in $\Dist$ of the coproduct $A_{10}+A_{11}+A_{12}$ of metric spaces $A_{10}=\{u,v\}$, $A_{11}=\{v',w\}$ and $A_{12}=\{u',w'\}$ modulo the equivalence relation $u\sim u'$, $v\sim v'$ and $w\sim w'$. Hence $FA_1$ is the quotient
of $A_{10}+A_{11}+A_{12}$ modulo this equivalence relation in $\Met$.
Consequently $FA$ is the quotient of the equivalence relation $\sim$
on $B=A_{00}+A_{01}+A_{02}+A_{10}+A_{11}+A_{12}$ in $\Met$. Similarly,
$F(A\times A)$ is the quotient of $\sim\times\sim$ on $B\times B$.

Example \ref{power} thus shows that quotients of equivalence relations
do not commute with finite powers in $\Met$. Hence reflexive coequalizers
do not commute with finite powers in $\Met$. Consequently, \cite[Corollary 5.2, Corollary 5.3, Lemma 5.4 and Example 5.6]{R} are false.  

\vskip 2mm
{\bf Competing interests.} The author declare none.

\end{document}